\documentclass[10pt]{article}
\usepackage[all]{xy}
\usepackage{amsmath}
\usepackage{amsfonts}
\usepackage{amssymb}
\usepackage{amscd}
\usepackage{amsthm}
\usepackage{latexsym}
\usepackage{amsbsy}
\newtheorem{lemma}{Lemma}[section]
\newtheorem{proposition}{Proposition}[section]
\newtheorem{theorem}{Theorem}[section]
\newtheorem{corollary}{Corollary}[section]
\newtheorem{definition}{Definition}[section]

\newcommand{\compl}{\mathbb{C}}
\newcommand{\qp}{\mathbb{C}P}
\newcommand{\ud}{\,\mathrm{d}}
\newcommand{\delb}{\overline{\partial}}
\begin{document}
\title{The homogeneous coordinate ring of the quantum projective plane}
\author {Masoud Khalkhali and Ali Moatadelro\\
 Department of Mathematics, University of Western Ontario\\
 London, Ontario, Canada}
\date{}
\maketitle{}
\begin{abstract}
 We define   holomorphic structures on canonical  line bundles on the quantum projective plane.  The space of
 holomorphic sections of these line bundles  will determine the
 quantum homogeneous coordinate ring of $\qp^2_q$. We also show that the
 holomorphic structure of $\qp^2_q$ is naturally represented by a
 twisted positive Hochschild 4-cocycle. 

\end{abstract}
\begin{center}\tableofcontents
\end{center}
\section{Introduction}
In this paper we continue a study of complex structures on quantum projective spaces that was initiated in \cite{KLS} for
$\qp^1_q$. In the present paper we consider a natural holomorphic structure on the quantum projective plane $\qp^2_q$ already presented in \cite{DDL,DL}, and define  holomorphic structures on canonical quantum line bundles on it. The space of holomorphic sections of these line bundles then will determine the quantum homogeneous coordinate ring of $\qp^2_q$.

In Section \ref{pre}, we review  basic notions of a complex structure on an involutive algebra as well as complex
structures on modules and bimodules over such algebra from \cite{KLS}. In Section \ref{qp2}, we recall the definition of
the  quantum projective plane $\qp^2_q$ \cite{DDL}, and its canonical line bundles. In Section \ref{Cstr}, we introduce
a flat $\delb$-connection on the bimodules representing canonical quantum line bundles on $\qp^2_q$ . We also
establish the compatibility
of these connections   with the natural tensor product of these bimodules. This compatibility is then used to derive the
structure of the quantum
homogeneous coordinate ring of $\qp^2_q$ as a twisted polynomial algebra in three variables. In Section \ref{holcts} we extend the results of Section 4 to $L^2$-functions and $L^2$-sections.  

In the last section, using the complex structure on $\qp^2_q$, we give a formula for a twisted Hochschild 4-cocycle on
$\mathcal A(\qp^2_q)$ cohomologous to its fundamental cyclic 4-cocycle which is  originally defined via its
smooth structure. We also show that this cocycle is twisted positive in an appropriate sense \cite{KLS}.
This fits well with
the point of view on holomorphic structures in noncommutative geometry advocated in \cite{C1,C2}.

\section{Preliminaries}\label{pre}
In this section we review the general setup of noncommutative complex structure on
a given $\ast$-algebra as introduced in \cite{KLS}.
\subsection{Noncommutative complex structures}
Let $\mathcal A$ be a $\ast$-algebra over $\mathbb C$. A \textit{differential $\ast$-calculus} for
 $\mathcal A$ is a pair $(\Omega^{\bullet}(\mathcal A),\ud)$, where
$\Omega^{\bullet}(\mathcal A)=\bigoplus_{n \geq 0}\Omega^n(\mathcal A)$
 is a graded differential $\ast$-algebra with $\Omega^0(\mathcal A)=\mathcal A$. The differential map
 $\ud:\Omega^n(\mathcal A)\rightarrow \Omega^{n+1}(\mathcal A)$ satisfying the graded Leibniz rule,
 $\ud(\omega_1\omega_2)=(\ud\omega_1)\omega_2+(-1)^n\omega_1(\ud\omega_2)$ and
 $\ud^2=0$. The differential also anti commutes with the $\ast$-structure:
 $\ud(a^*)=-(\ud a)^*$.
\begin{definition}
A complex structure on an algebra $\mathcal A$, equipped with a differential calculus
 $(\Omega^{\bullet}(\mathcal A),\ud)$,
is a bigraded differential $\ast$-algebra $\Omega^{(\bullet,\bullet)}(\mathcal A)$
and two differential
maps $\partial:\Omega^{(p,q)}(\mathcal A)\rightarrow \Omega^{(p+1,q)}(\mathcal A)$ and
$\delb:\Omega^{(p,q)}(\mathcal A)\rightarrow \Omega^{(p,q+1)}(\mathcal A)$ such that:
\begin{eqnarray}
\Omega^n(\mathcal A)=\bigoplus_{p+q=n}\Omega^{(p,q)}(\mathcal A)\,,\quad \partial a^*=-(\delb a)^*\,,
 \quad \ud=\partial+\delb.
\end{eqnarray}
Also, the involution $\ast$ maps $\Omega^{(p,q)}(\mathcal A)$ to $\Omega^{(q,p)}(\mathcal A)$.
\end{definition}
We will use the simple notation $(\mathcal A,\delb)$ for a complex structure on $\mathcal A$.
\begin{definition}
Let $(\mathcal A,\delb)$ be an algebra with a complex structure. The space of holomorphic
elements of $\mathcal A$ is defined as
\begin{equation}
\mathcal O(\mathcal A):=Ker\{\delb:\mathcal A\rightarrow\Omega^{(0,1)}(\mathcal A)\}.\nonumber
\end{equation}
\end{definition}
\subsection{Holomorphic connections}
Suppose we are given a differential calculus $(\Omega^{\bullet}(\mathcal A),\ud)$. We recall that a
connection on a left $\mathcal A$-module $\mathcal E$ for the differential
calculus $(\Omega^{\bullet}(\mathcal A),\ud)$ is a linear map
$\nabla:\mathcal E\rightarrow\Omega^1(\mathcal A)\otimes_{\mathcal A} \mathcal E$
with left Leibniz property:
\begin{equation}\label{left Leibniz}
\nabla(a\xi)=a\nabla \xi+\ud a \otimes_{\mathcal A} \xi,\quad \forall a\in \mathcal A,\,
\forall \xi \in \mathcal E.
\end{equation}
By the graded Leibniz rule, i.e.
\begin{equation}
\nabla(\omega\xi)=(-1)^n\omega \nabla \xi+\ud \omega \otimes \xi,
 \quad \forall\omega \in \Omega^n(\mathcal A),\, \forall \xi \in \Omega(\mathcal A)\otimes_{\mathcal A} \mathcal E,
\end{equation}
this connection can be uniquely extended to a map, which will be denoted again by $\nabla$, $\nabla:
\Omega^{\bullet}(\mathcal A) \otimes_{\mathcal A} \mathcal E
\rightarrow \Omega^{\bullet+1}(\mathcal A) \otimes_{\mathcal A} \mathcal E$.

The curvature of such a connection is defined by $F_\nabla=\nabla \circ \nabla$.
One can show that, $F_\nabla$ is an element of
Hom$_\mathcal A (\mathcal E,\Omega^2(\mathcal A)\otimes_{\mathcal A} \mathcal E)$.
\begin{definition}\label{holomorphic vect bundle}
Suppose $(\mathcal A,\delb)$ is an algebra with a complex structure.
 A holomorphic structure on a left $\mathcal A$-module
$\mathcal E$ with respect to this complex structure is given by a linear map
$\nabla^{\delb}:\mathcal E \rightarrow \Omega^{(0,1)} \otimes_{\mathcal A} \mathcal E$ such that
\begin{eqnarray}
&&\nabla^{\delb}(a\xi)=a\nabla^{\delb}\xi+\delb a \otimes_{\mathcal A} \xi, \quad \forall
a \in \mathcal A,\, \forall \xi \in \mathcal E,
\end{eqnarray}
and such that $F_{\nabla^{\delb}}=(\nabla^{\delb})^2=0$.
\end{definition}
Such a connection will be called a flat $\delb$-connection. In the case which $\mathcal E$ is a finitely generated $\mathcal A$-module, $(\mathcal E,\nabla^{\delb})$ will be called a holomorphic vector bundle.

Associated to a flat $\delb$-connection, there exists a complex of vector spaces
\begin{equation}
0\rightarrow \mathcal E \rightarrow \Omega^{(0,1)} \otimes_{\mathcal A} \mathcal E
 \rightarrow \Omega^{(0,2)} \otimes_{\mathcal A} \mathcal E\rightarrow...
\end{equation}
Here $\nabla^{\delb}$ is extended to $\Omega^{(0,q)} \otimes_{\mathcal A} \mathcal E$
by the graded Leibniz rule. The zeroth cohomology group of this complex is called the space of Holomorphic sections of
$\mathcal E$ and will be denoted by $H^0(\mathcal E,\nabla^{\delb})$.
\subsection{Holomorphic structures on bimodules}
\begin{definition}
Let $\mathcal A$ be an algebra with a differential calculus $(\Omega^{\bullet}(\mathcal A),\ud)$.
A bimodule connection on an $\mathcal A$-bimodule $\mathcal E$ is given by a connection $\nabla$
 which satisfies a left Leibniz rule as in formula ({\ref{left Leibniz}})
and a right $\sigma$-twisted Leibniz property with respect to a bimodule isomorphism
$\sigma:\mathcal E\otimes_{\mathcal A} \Omega^1(\mathcal A) \rightarrow
\Omega^1(\mathcal A)\otimes_{\mathcal A} \mathcal E$. i.e.
 \begin{equation}
\nabla(\xi a)=\nabla \xi a+\sigma(\xi\otimes \ud a)\,, \quad\forall \xi \in\mathcal E,\,
\forall a \in \mathcal A.
\end{equation}
\end{definition}
The tensor product connection of two bimodule connections $\nabla_1$ and $\nabla_2$
on two $\mathcal A$-bimodules $\mathcal E_1$ and $\mathcal E_2$ with respect to the bimodule isomorphisms
$\sigma_1$ and $\sigma_2$ is a map $\nabla:\mathcal E_1\otimes_{\mathcal A} \mathcal E_2\rightarrow
\Omega^1(\mathcal A)\otimes_{\mathcal A}\mathcal E_1\otimes_{\mathcal A} \mathcal E_2$ defined by
\begin{equation}
\nabla:=\nabla_1 \otimes 1 +(\sigma_1\otimes 1)(1\otimes \nabla_2).\nonumber
\end{equation}
It can be checked that, $\nabla$ has the right $\sigma$-twisted property with
$\sigma:\mathcal E_1\otimes \mathcal E_2\otimes \Omega^1(\mathcal A)\rightarrow
\Omega^1(\mathcal A)\otimes\mathcal E_1\otimes \mathcal E_2$ given by
$\sigma=(\sigma_1\otimes 1)\circ(1\otimes\sigma_2)$.
\section{The quantum projective plane $\qp^2_q$}\label{qp2}
In this section, we recall the definition of the quantum enveloping algebra
$U_q(\mathfrak{su}(3))$, the quantum group $\mathcal A(SU_q(3))$ and the pairing between them. We also recall the definition of the quantum projective plane $\qp^2_q$ and its canonical quantum line bundles \cite{DDL}.
\subsection {The quantum enveloping algebra $U_q(\mathfrak{su}(3))$}
Let $0<q<1$. We use the following notation
\begin{align*}
[a, b]_q=ab-q^{-1}ba,\,[z]=\frac{q^z-q^{-z}}{q-q^{-1}},\,\begin{bmatrix}n\\m\end{bmatrix}&=\frac{[n]!}{[m]![n-m]!}\,,\\
[j,k,l]!=q^{-(jk+kl+lj)}\frac{[j+k+l]!}{[j]![k]![l]!}\,.&
\end{align*}

The Hopf $\ast$-algebra $U_q(\mathfrak{su}(3))$ as a $\ast$-algebra is generated by
$K_i,K^{-1}_i,E_i,F_i$, $i=1,2$ with $K_i^*=K_i$, $E_i^*=F_i$
subject to the relations
\begin{eqnarray}
&&[K_i,K_j]=0 , \quad K_iE_i=qE_iK_i , \quad [E_i,F_i]=
(q-q^{-1})^{-1}(K_i^2-K_i^{-2}),\nonumber\\
&&K_iE_j=q^{-1/2}E_jK_i , \quad[E_i,F_j]=0 , \quad i\neq j\,,\nonumber
\end{eqnarray}
and
\begin{equation}
E_i^2E_j+E_jE_i^2=(q+q^{-1})E_iE_jE_i \quad i\neq j\,. \nonumber
\end{equation}
Its coproduct, counit and antipode are defined on generators as
\begin{eqnarray}
&&\Delta(E_i)=E_i\otimes K_i+K_i^{-1}\otimes E_i,
\quad \Delta(F_i)=F_i\otimes K_i+K_i^{-1}\otimes F_i,\nonumber\\
&&\Delta(K_i)=K_i\otimes K_i,\quad \epsilon(K_i)=1,\quad \epsilon
(E_i)=\epsilon(F_i)=0,\nonumber\\
&&S(K_i)=K_i^{-1},\quad S(E_i)=-qE_i, \quad S(F_i)=-q^{-1}F_i.\nonumber
\end{eqnarray}

Let $V(n_1,n_2)$ be the irreducible finite dimensional $*$-representation
of $U_q(\mathfrak{su}(3))$ \cite{KS} with the orthonormal basis $|n_1,n_2,j_1,j_2,m\rangle$,
where indices are restricted by\\
\begin{equation}\label{j indices}
j_i=0,1,2,...,n_i, \quad \frac{1}{2}(j_1+j_2)-|m|\in\mathbb N.
\end{equation}
The generators of $U_q(\mathfrak{su}(3))$ act on this basis as
\begin{eqnarray}\label{action on basis}
&&K_1|n_1,n_2,j_1,j_2,m\rangle=q^m|n_1,n_2,j_1,j_2,m\rangle,\nonumber\\
&&K_2|n_1,n_2,j_1,j_2,m\rangle=q^{\frac{3}{4}(j_1-j_2)+\frac{1}{2}(n_2-n_1-m)}|n_1,n_2,j_1,j_2,m\rangle,
\nonumber\\
&&E_1|n_1,n_2,j_1,j_2,m\rangle=\sqrt{[\frac{1}{2}(j_1+j_2)-m][\frac{1}{2}(j_1+j_2)+m+1]}
\nonumber\\
&&\qquad\qquad\qquad|n_1,n_2,j_1,j_2,m+1\rangle,\nonumber\\
&&E_2|n_1,n_2,j_1,j_2,m\rangle=\sqrt{[\frac{1}{2}(j_1+j_2)-m+1]}A_{j_1,j_2}
|n_1,n_2,j_1+1,j_2,m-\frac{1}{2}\rangle\nonumber\\
&&\qquad+\sqrt{[\frac{1}{2}(j_1+j_2)+m]}B_{j_1,j_2}|n_1,n_2,j_1,j_2-1,m-\frac{1}{2}\rangle,
\end{eqnarray}
where
\begin{eqnarray}\label{Aij Bij}
&&A_{j_1,j_2}:=\sqrt{\frac{[n_1-j_1][n_2+j_1+2][j_1+1]}{[j_1+j_2+1][j_1+j_2+2]}},
\\
&&B_{j_1,j_2}:=
\begin{cases}
\sqrt{\frac{[n_1+j_2+1][n_2-j_2+1][j_2]}{[j_1+j_2][j_1+j_2+1]}} &
 \text {if} \, j_1+j_2\neq 0,\\
1 & \text {if}\, j_1+j_2=0.
\end{cases}
\end{eqnarray}
\subsection {The quantum group $\mathcal A(SU_q(3))$}\label{qgsuq3}
As a $\ast$-algebra, $\mathcal A(SU_q(3))$ is generated by $u^i_j$, $i,j=1,2,3$,
satisfying the following commutation relations
\begin{align*}
&u^i_ku^j_k=qu^j_ku^i_k, &u^k_iu^k_j=qu^k_ju^k_i \quad \forall\, i<j,&\\
&[u^i_l,u^j_k]=0, &[u^i_k,u^j_l]=(q-q^{-1})u^i_lu^j_k &\quad \forall\, i<j,\,k<l,
\end{align*}
and a cubic relation
\begin{eqnarray}
\sum_{\sigma \in S_3}(-q)^{l(\sigma)}u^1_{\sigma(1)}u^2_{\sigma(2)}u^3_{\sigma(3)}=1.\nonumber
\end{eqnarray}
In the last equation, sum is taken over all permutation $\sigma$ on three letters and $l(\sigma)$
 is the length of $\sigma$. The involution $*$ is deafened as
\begin{eqnarray}
(u^i_j)^*:=(-q)^{j-i}(u^{k_1}_{l_1}u^{k_2}_{l_2}-qu^{k_1}_{l_2}u^{k_2}_{l_1}),
\end{eqnarray}
where as an ordered set, $\{k_1,k_2\}=\{1,2,3\}\setminus\{i\}$ and $\{l_1,l_2\}=\{1,2,3\}\setminus\{j\}$.
The Hopf algebra structure is given by
\begin{eqnarray}
\Delta(u^i_j)=\sum_k u^i_k\otimes u^k_j,\quad \epsilon(u^i_j)=\delta^i_j,\quad
S(u^i_j)=(u^j_i)^*. \nonumber
\end{eqnarray}

There exists a non-degenerate pairing between Hopf algebras $\mathcal A(SU_q(3))$ and $U_q(\mathfrak{su}(3))$,
which allows us to define a left and a right action of $U_q(\mathfrak{su}(3))$ on $\mathcal A(SU_q(3))$.
These actions make $\mathcal A(SU_q(3))$ an $U_q(\mathfrak{su}(3))$-bimodule $\ast$-algebra.\\
The actions are defined as
\begin{eqnarray}
h\triangleright a=a_{(1)}\langle h,a_{(2)}\rangle, \quad a\triangleleft h=\langle h,a_{(1)}\rangle a_{(2)}.\nonumber
\end{eqnarray}
Here we used Sweedler's notation. Left and right actions on generators are given by
(see \cite{DL})
\begin{align}\label{lr actions}
& K_i \triangleright u^j_k=q^{\frac{1}{2}(\delta_{i+1,k}-\delta_{i,k})}u^j_k,
\quad E_i \triangleright u^j_k=\delta_{i,k}u^j_{i+1},
\quad F_i \triangleright u^j_k=\delta_{i+1,k}u^j_{i},\nonumber\\
& u^j_k \triangleleft K_i=q^{\frac{1}{2}(\delta_{i+1,j}-\delta_{i,j})}u^j_k,
\quad u^j_k \triangleleft E_i=\delta_{i+1,j}u^i_k,
\quad u^j_k \triangleleft F_i=\delta_{i,j}u^{i+1}_{k}.
\end{align}

A linear basis of $\mathcal A(SU_q(3))$ corresponding to the Peter-Weyl decomposition is given by (see \cite{DL,DDL})
\begin{equation}\label{t basis}
t(n_1,n_2)^{l_1,l_2,k}_{j_1,j_2,m}:=
X^{n_1,n_2}_{j_1,j_2,m}
\triangleright \{(u^1_1)^*\}^{n_1}(u^3_3)^{n_2}\triangleleft (X^{n_1,n_2}_{l_1,l_2,k})^*.
\end{equation}
where 
$X^{n_1,n_2}_{j_1,j_2,m}$ is defined as
\begin{eqnarray}
&&X^{n_1,n_2}_{j_1,j_2,m}:=
N^{n_1,n_2}_{j_1,j_2,m}\nonumber\\&&\sum_{k=0}^{n_1-j_1}
\frac{q^{-k(j_1+j_2+k+1)}}{{[j_1+j_2+k+1]!}}
\begin{bmatrix}n_1-j_1\\
k\end{bmatrix}F_1^{1/2(j_1+j_2)-m+k}
[F_2,F_1]_q^{n_1-j_1-k}F_2^{j_2+k}.\nonumber
\end{eqnarray}
The coefficients $N^{n_1,n_2}_{j_1,j_2,m}$ are defined by
\begin{equation}
N^{n_1,n_2}_{j_1,j_2,m}=\sqrt{[j_1+j_2+1]}\sqrt{\frac{[\frac{j_1+j_2}{2}+m]!
[n_2-j_2]![j_1]![n_1+j_2+1]![n_2+j_1+1]!}{[\frac{j_1+j_2}{2}-m]![n_1-j_1]![j_2]!
[n_1]![n_2]![n_1+n_2+1]!}}.\nonumber
\end{equation}

The Peter-Weyl isomorphism $Q:\mathcal A(SU_q(3))\rightarrow\bigoplus_{\substack{(n_1,n_2)}} V(n_1,n_2)\otimes V(n_1,n_2)$
has the following property for all $h\in U_q(\mathfrak{su}(3))$:
\begin{align}\label{Q iso}
Q(h\triangleright t(n_1,n_2)^{l_1,l_2,k}_{j_1,j_2,m}) &=h|n_1,n_2,j_1,j_2,m\rangle\otimes |n_1,n_2,l_1,l_2,k\rangle,\nonumber\\
Q(t(n_1,n_2)^{l_1,l_2,k}_{j_1,j_2,m}\triangleleft h) &=|n_1,n_2,j_1,j_2,m\rangle\otimes\, \theta(h)|n_1,n_2,l_1,l_2,k\rangle,
\end{align}
where $\theta:U_q(\mathfrak{su}(3))\rightarrow U_q(\mathfrak{su}(3))^{op}$ is the Hopf $*$-algebra isomorphism which is defined on generators as
\begin{equation}
\theta(K_i)=K_i,\,\theta(E_i)=F_i,\,\theta(F_i)=E_i,\nonumber
\end{equation}
and satisfying $\theta^2=id$.

We define the quantum projective plane $\qp^2_q$ as a quotient of the 5-dimensional quantum sphere (\cite{DDL}). By definition
\begin{equation}
\mathcal A(S^5_q):=\{a \in \mathcal A(SU_q(3)) |\,a \triangleleft h=\epsilon(h)a \,,\,\forall h\in U_q(\mathfrak{su}(2))\}.\nonumber
\end{equation}
As a $\ast$-algebra, $\mathcal A(S^5_q)$ is generated by elements $ z_j=u^3_j,\,j=1,2,3$ of $\mathcal A(SU_q(3))$. Abstractly, this algebra is the algebra with generators $z_i,z_i^*$ $i=1,2,3$ and subject to the following relations
\begin{align*}
&z_iz_j=qz_jz_i\quad \forall\, i<j, &z^*_iz_j=qz_jz^*_i,\quad\forall\, i\neq j,\\
&[z_1^*,z_1]=0, &[z_2^*,z_2]=(1-q^2)z_1z^*_1,\\
&[z_3,z_3]=(1-q^2)(z_1z_1^*+z_2z_2^*),
&z_1z_1^*+z_2z_2^*+z_3z_3^*=1.
\end{align*}

Now we define the algebra $\mathcal A(\qp^2_q)$ of the quantum projective plane as a $\ast$-subalgebra of $\mathcal A(S^5_q)$.
\begin{equation}
\mathcal A(\qp^2_q):=\{a\in \mathcal A(S^5_q)|\, a\triangleleft K_1K_2^2=a\}.\nonumber
\end{equation}
One can show that \cite{DDL}, $\mathcal A(S^5_q)\simeq\bigoplus_{(n_1,n_2)\in\mathbb N^2} V(n_1,n_2)$
with the basis $t(n_1,n_2)^{\underline 0}_{\underline j}$, where $n_1$ and $n_2$ are non-negative integers.
 Also $\mathcal A(\qp^2_q)\simeq\bigoplus_{n\in\mathbb N}V(n,n)$ with the basis $t(n,n)^{\underline 0}_{\underline j}$.
Here we have used the multi index notation $\underline j=j_1,j_2,m$ and indices $j_1,j_2,m$ are restricted by (\ref{j indices}).

For any integer $N$, we define the space of the canonical quantum line bundle $\mathcal L_N$ on $\qp^2_q$ by
\begin{equation}
\mathcal L_N:=\{a\in \mathcal A(S^5_q) : a \triangleleft K_1K_2^2=q^Na\}.\nonumber
\end{equation}
These spaces are $\mathcal A(\qp^2_q)$-bimodules. One can see that \cite{DDL},
\begin{eqnarray}
\mathcal L_N=\bigoplus_{n\in \mathbb N}V(n,n+N)\quad\text{if}\,\,N\geq 0,\,\,\text{and}\,\,
\mathcal L_N=\bigoplus_{n \in \mathbb N}V(n-N,n)\quad\text{if}\,\, N<0.\nonumber
\end{eqnarray}
The basis elements are given by $t(n,n+N)^{\underline 0}_{\underline j}$ for $N\geq 0$ and $t(n-N,n)^{\underline 0}_{\underline j}$ for $N<0$.
\section{The complex structure of $\qp^2_q$}\label{Cstr}
There is a complex structure on $\qp^2_q$ defined in \cite{DL,DDL}.
For future use, we give an explicit description of the spaces $\Omega^{(0,0)}$, $\Omega^{(0,1)}$ and $\Omega^{(0,2)}$:
\begin{equation}
\Omega^{(0,0)}=\mathcal L_0=\mathcal A(\qp^2_q), \quad \Omega^{(0,2)}=\mathcal L_3,\nonumber
\end{equation}
and as a subspace of $\mathcal A(SU_q(3))^2$, $\Omega^{(0,1)}$ contains all pairs
$(v_+,v_-)$ such that the following conditions hold
\begin{align}\label{omega01}
&(v_+,v_-)\triangleleft K_1K_2^2=q^{\frac{3}{2}}(v_+,v_-),&
&(v_+,v_-)\triangleleft K_1=(q^{\frac{1}{2}}v_+,q^{-\frac{1}{2}}v_-),\nonumber\\
&(v_+,v_-)\triangleleft F_1=(0,v_+),&
&(v_+,v_-)\triangleleft E_1=(v_-,0).
\end{align}
The complex structure on $\qp^2_q$ is given by the maps
$\partial:\mathcal A(\qp^2_q)\rightarrow\Omega^{(1,0)}(\qp^2_q)$ and
$\delb:\mathcal A(\qp^2_q)\rightarrow\Omega^{(0,1)}(\qp^2_q)$, which (up to multiplicative constants) are
$\partial a=(a\triangleleft E_2 , a \triangleleft F_2E_1)^t$,
$\delb a=(a\triangleleft F_2F_1 , a \triangleleft F_2)^t$.

In this section we identify the space of holomorphic functions on $\qp^2_q$ and holomorphic sections of $\mathcal L_N$.
\subsection{Holomorphic functions}
\begin{proposition}\label{holfunprop}
There are no non-trivial holomorphic polynomials on $\qp^2_q$.
\end{proposition}
\begin{proof}
Let $a=\sum_{n,\underline j} \lambda_{n,\underline j}t(n,n)^{\underline 0}_{\underline j}$. Then $\delb a=0$ implies that
$a\triangleleft F_2=0$ and $a \triangleleft F_2F_1=0$.
A simple computation shows that
$a\triangleleft F_2=\sum \lambda_{n,\underline j}\gamma_nt(n,n)^{1,0,-\frac{1}{2}}_{\underline j}$, where
$\gamma_n=A_{0,0}=\sqrt{\frac{[n][n+2]}{[2]}}$.
This can be obtained by (\ref{Q iso}), (\ref{action on basis}) and (\ref{Aij Bij}) because\\
$E_2|n,n,0,0,0\rangle=A_{0,0}|n,n,1,0,-\frac{1}{2}\rangle.$\\
Since
$\gamma_n=0$ iff $n=0$,
all coefficients need to be zero except $c_{0,0}$. Note that the action of $F_1$
does not put more restrictions on the coefficients. This demonstrates that
\begin{eqnarray}
Ker\{\delb : \mathcal A(\qp^2_q)\rightarrow\Omega^{(0,1)}(\qp^2_q)\}=\langle t(0,0)_{\underline 0}^{\underline 0}\rangle=\mathbb{C}.\nonumber
\end{eqnarray}
\end{proof}
This preposition, already has been proved in \cite{DDL} as a result of a Hodge decomposition.
\subsection{Canonical line bundles}
Like \cite{DL}, we define the connection $\nabla_N$ on $\mathcal L_N$ by $\nabla_N:=q^{-N}\Psi_N^\dagger\ud \Psi_N$, where $\Psi_N$ is the column vector with components $\psi^N_{i,j,k}$ given by
\begin{align*}
&(\psi^N_{j,k,l})^*=\sqrt{[j,k,l]!}z_1^{j}z_2^{k}z_3^{l},&\text{if}\,
N\geq 0 \quad\text{and with}\,j+k+l=N,\\
&(\psi^N_{j,k,l})^*=\sqrt{[j,k,l]!}(z_1^{j}z_2^{k}z_3^{l})^*,&\text{if}\,
N\leq 0\quad \text {and with}\,i+j+k=-N.
\end{align*}
Notice that we put an extra coefficient $q^{-N}$. This is needed for compatibility with the twist map in section (\ref{bimod connection}).

The anti holomorphic part of this connection will be
 $\nabla_N^{\delb}=q^{-N}\Psi_N^\dagger\delb\Psi_N$.
The curvature of $\nabla_N^{\delb}$ can be computed as follows
\begin{eqnarray}
(\nabla_N^{\delb})^2=q^{-2N}\Psi_N^\dagger (\delb P_N\delb P_N)\Psi_N,
\nonumber
\end{eqnarray}
where $P_N:=\Psi_N\Psi_N^{\dagger}$ is a projection map due to the fact that
$\Psi_N^{\dagger}\Psi_N=1$.
\begin{proposition}\label{flatness}
The connection $\nabla^{\delb}_N$ is flat.
\end{proposition}
\begin{proof}
We will prove this for $N\geq 0$ and a similar discussion will cover the case $N<0$.

It suffices to show that
\begin{align}
\Psi_N^{\dagger}\delb P_N=\Psi_N^{\dagger}(P_N\triangleleft F_2F_1,P_N\triangleleft F_2)^t=0.\nonumber
\end{align}
The second component
\begin{align}
\Psi_N^{\dagger}(P_N\triangleleft F_2)&=\Psi_N^{\dagger}((\Psi_N\Psi_N^{\dagger})
 \triangleleft F_2)\nonumber\\
&=\Psi_N^{\dagger}\{(\Psi_N\triangleleft F_2)(\Psi_N^{\dagger}\triangleleft K_2)+(\Psi_N\triangleleft K_2^{-1})(\Psi_N^{\dagger}\triangleleft F_2)\}\nonumber\\
&=0\,.\nonumber
\end{align}
and this last equality is obtained by ( see \cite{DL}, section 6)
\begin{equation}\label{psi equalities 1}
\Psi_N^\dagger\triangleleft F_2=0\,,\quad \Psi_N^\dagger(\Psi_N\triangleleft F_2)=0\,.
\end{equation}
Similar computation shows that $\Psi_N^{\dagger}(P_N\triangleleft F_2F_1)$ also vanishes. For this the following identity is needed.
\begin{equation}\label{psi equalities 2}
\Psi_{N}^{\dagger}(\Psi_{N}\triangleleft F_2F_1)=0.
\end{equation}
Hence $(\nabla^{\delb}_N)^2=0$.
\end{proof}
Alternatively, as it was kindly pointed out to us by Francesco D'Andrea, using Lemma 6.1 in \cite{DL},  the full connection (holomorphic + antiholomorphic part) has curvature of type $(1, 1)$. This implies that the square of the holomorphic and antiholomorphic part is zero.

Proposition (\ref{flatness}) verifies that the operator $\nabla_N^{\delb}$ satisfies the condition of holomorphic structure as
given in the definition (\ref{holomorphic vect bundle}).

Flatness of $\nabla^{\delb}_N$ gives the following complex of vector spaces
\begin{eqnarray}
0\rightarrow \mathcal L_N \rightarrow \Omega^{(0,1)}\otimes_{\mathcal A(\qp^2_q)} \mathcal L_N\rightarrow
\Omega^{(0,2)}\otimes_{\mathcal A(\qp^2_q)} \mathcal L_N\rightarrow 0\,.\nonumber
\end{eqnarray}
The zeroth cohomology group $H^0(\mathcal L_N,\nabla^{\delb}_N)$ of this complex is called the
\textit {space of holomorphic sections of $\mathcal L_N$}.
The structure of this space is best described by the following theorem.
\begin{theorem}\label{holsectthm}
Let N be a positive integer. Then
\begin{eqnarray}
(1) &&\qquad H^0(\mathcal L_N,\nabla_N^{\delb})\simeq\mathbb C^{\frac{(N+1)(N+2)}{2}}\nonumber\\
(2) &&\qquad H^0(\mathcal L_{-N},\nabla_{-N}^{\delb})=0.\nonumber
\end{eqnarray}
\end{theorem}
\begin{proof}
First we recall that
\begin{align}
\nabla_N^{\delb}\xi=q^{-N}\Psi_N^\dagger\delb\Psi_N\xi=q^{-N}\Psi^\dagger_N
((\Psi_N\xi)\triangleleft F_2F_1 ,(\Psi_N\xi)\triangleleft F_2)^t.\nonumber
\end{align}
Using (\ref{psi equalities 1}), (\ref{psi equalities 2}) and the following identities
\begin{align}\label{psi equalities 3}
\Psi_{N}\triangleleft F_1=0,\quad \Psi_{N}\triangleleft K_1=\Psi_{N},\quad \Psi_{N}\triangleleft
 K_2=q^{-N/2}\Psi_{N},
\end{align}
we prove that $\nabla_N^{\delb} \xi=0$ is equivalent to the equations
$\xi\triangleleft F_2=0$ and $\xi\triangleleft F_2F_1=0$.

First we compute the second component of $\nabla_N^{\delb}\xi$.
\begin{align}
q^{-N}\Psi_N^{\dagger}((\Psi_N\xi)\triangleleft F_2)&=q^{-N}\Psi_N^{\dagger}
\{(\Psi_N\triangleleft F_2)(\xi\triangleleft K_2)+(\Psi_N\triangleleft K_2^{-1}
)(\xi\triangleleft F_2)\}\nonumber\\
&=q^{-N/2}\xi\triangleleft F_2.\nonumber
\end{align}
In addition to (\ref{psi equalities 1}) and (\ref{psi equalities 3}), here we have used $\Psi_N^{\dagger}\Psi_N=1$.
In a similar manner, one can show that the first component is
\begin{align}
q^{-N}\Psi_N^{\dagger}((\Psi_N\xi)\triangleleft &F_2F_1)\nonumber\\
&=q^{-N}\Psi_N^{\dagger}
\{(\Psi_N\triangleleft F_2)(\xi\triangleleft K_2)+(\Psi_N\triangleleft K_2^{-1}
)(\xi\triangleleft F_2)\}\triangleleft F_1\nonumber\\
&=q^{-N}\Psi_N^{\dagger}
\{q^{N/2}(\Psi_N\triangleleft F_2)\xi+q^{N/2}\Psi_N(\xi\triangleleft F_2)\}\triangleleft F_1\nonumber\\
&=q^{-N/2}\Psi_N^{\dagger}
\{(\Psi_N\triangleleft F_2F_1)(\xi\triangleleft K_1)+(\Psi_N\triangleleft F_2K_1^{-1})(\xi\triangleleft F_1)\nonumber\\&+(\Psi_N\triangleleft F_1)(\xi\triangleleft F_2K_1)+(\Psi_N\triangleleft K_1^{-1})(\xi\triangleleft F_2F_1)\}\nonumber\\
&=q^{-N/2}\xi\triangleleft F_2F_1\nonumber.
\end{align}

Let $N\geq 0$. In this case, a basis element of $\mathcal L_N$ is of the form $t(n,n+N)^{\underline 0}_{\underline j}$. Similar computation to the proof of proposition \ref{holfunprop}, using
(\ref{Q iso}), (\ref{action on basis}) and (\ref{Aij Bij}), shows that
$t(n,n+N)^{\underline 0}_{\underline j}\triangleleft F_2=
\gamma_nt(n,n+N)^{1,0,-\frac{1}{2}}_{\underline j}$,
where $\gamma_n=A_{0,0}=(\frac{[n][n+N+2]}{[2]})^{1/2}$.
If $\xi\in\mathcal L_N$, then $\xi$ can be written as $\sum_{n,\underline j} \lambda_{n,\underline j} t(n,n+N)^{\underline 0}_{\underline j}$. So
$\xi\triangleleft F_2=\sum\lambda_{n,\underline j}\gamma_nt(n,n+N)^{1,0,-1/2}_{\underline j}$.
Since $\gamma_n=0$ iff $n=0$,
$\xi\triangleleft F_2=0$ implies that the set $\{t(0,N)^{\underline 0}_{\underline j}\}$
will form a basis for the space of
Ker$ \nabla_N^{\delb}$.
Remembering that by (\ref{j indices}), the indices are restricted by $j_1=0,j_2=0,...,N$, and $j_2/2-|m|\in \mathbb N$, we
will find that dim Ker $\nabla_N^{\delb}=\frac{(N+1)(N+2)}{2}$.

When N is a negative integer, $\gamma_n$ will be $(\frac{[n-N][n+2]}{[2]})^{1/2}$ which is nonzero.
So dim Ker $\nabla_N^{\delb}=0$.

\end{proof}
\subsection{Bimodule connections}\label{bimod connection}
There exists a $\mathcal A(\qp^2_q)$-bimodules isomorphism $\sigma:\Omega^{(0,1)}\otimes_{\mathcal A(\qp^2_q)} \mathcal L_N\rightarrow
\mathcal L_N \otimes_{\mathcal A(\qp^2_q)} \Omega^{(0,1)}$ which acts as
\begin{eqnarray}
\sigma(\omega \otimes \xi)=q^{-N}\xi'\otimes \omega',\nonumber
\end{eqnarray}
such that both elements $\omega \otimes \xi$ and $\xi'\otimes \omega'$ in $\mathcal A(SU_q(3))^2$,
after multiplication are the same.
We try to illustrate this in the case of $N=1$. More precisely let us define the maps $\phi_1$ and $\phi_2$ as follows:
\begin{align*}
&\phi_1:\Omega^{(0,1)}\otimes_{\mathcal A(\qp^2_q)}\mathcal L_1 \rightarrow \mathcal A(SU_q(3))^2,\\
&\phi_1((v_+,v_-)^t\otimes \xi)=q^{\frac{1}{2}}(v_+\xi,v_-\xi)^t,\\
\end{align*}
and
\begin{align*}
&\phi_2:\mathcal L_1\otimes_{\mathcal A(\qp^2_q)}\Omega^{(0,1)} \rightarrow \mathcal A(SU_q(3))^2,\\
&\phi_2(\xi\otimes(v_+,v_-)^t)=q^{-\frac{1}{2}}(\xi v_+,\xi v_-)^t.
\end{align*}

We will prove that Im $\phi_1$=Im $\phi_2$. Therefore $\sigma=\phi_1^{-1} \phi_2$ gives an
isomorphism from $\mathcal L_1\otimes_{\mathcal A(\qp^2_q)}\Omega^{(0,1)}$ to
$\Omega^{(0,1)}\otimes_{\mathcal A(\qp^2_q)}\mathcal L_1$ which is coming from the multiplication map.
Let us first recall that as a $\ast$-algebra $\mathcal A(\qp^2_q)$ is generated by elements $p_{jk}=z_j^*z_k=(u^3_j)^*u^3_k$.
\begin{lemma}
With above notation $Im \,\phi_1=Im \,\phi_2$.
\end{lemma}
\begin{proof}
case1. $\alpha \in$ Im $\phi_2$ is a basis element.
\begin{align}
\alpha&=\phi_2(t(n,n+1)^{\underline 0}_{\underline i}\otimes p_{rs}\delb p_{jk})=
q^{-1/2}t(n,n+1)^{\underline 0}_{\underline i}
p_{rs}\binom{-q^{-3/2}(u^1_j)^*}{q^{-1/2}(u^2_j)^*} u^3_k \nonumber\\
&=q^{-1/2}\binom{-q^{-3/2}t(n,n+1)^{\underline 0}_{\underline i}p_{rs}(u^1_j)^*}{q^{-1/2}t(n,n+1)^{\underline 0}_{\underline i}p_{rs}(u^2_j)^*}u^3_k=q^{-1}\phi_1(T_{\underline irsj}\otimes u^3_k),\nonumber
\end{align}
where
\begin{equation}
T_{\underline irsj}=(-q^{-3/2}t(n,n+1)^{\underline 0}_{\underline i}p_{rs}(u^1_j)^*,q^{-1/2}p_{rs}t(n,n+1)^{\underline 0}_{\underline i}(u^2_j)^*)^t.\nonumber
\end{equation}
Since $u^3_k \in \mathcal L_1$, it is enough to prove that $T_{\underline irsj} \in \Omega^{(0,1)}$ .
In order to do so, we need to show that the pair $(v_+,v_-)$
 defined as below, satisfies the properties given in (\ref{omega01}).
\begin{eqnarray}
(v_+,v_-)^t=(-q^{-3/2}t(n,n+1)^{\underline 0}_{\underline i}p_{rs}(u^1_j)^*,
\,q^{-1/2}t(n,n+1)^{\underline 0}_{\underline i}p_{rs}(u^2_j)^* )^t.\nonumber
\end{eqnarray}

We will check $(v_+,v_-)\triangleleft E_1=(v_-,0)$.
\begin{align}
v_+\triangleleft E_1&=-q^{-3/2}t(n,n+1)^{\underline 0}_{\underline i}p_{rs}(u^1_j)^*\triangleleft E_1\nonumber\\
&=-q^{-3/2}\{(t(n,n+1)^{\underline 0}_{\underline i}\triangleleft E_1)
((p_{rs}(u^1_j)^*)\triangleleft K_1)\nonumber\\&+(t(n,n+1)^{\underline 0}_{\underline i}\triangleleft K_1^{-1})((p_{rs}(u^1_j)^*)\triangleleft E_1)\}\nonumber\\
&=-q^{-3/2}t(n,n+1)^{\underline 0}_{\underline i}
 \{(p_{rs}\triangleleft E_1)( (u^1_j)^*\triangleleft K_1)\nonumber\\
&+
(p_{rs}\triangleleft K_1^{-1})((u^1_j)^*\triangleleft E_1)\}\nonumber\\
&=-q^{-3/2}t(n,n+1)^{\underline 0}_{\underline i}p_{rs}(-q)(u^2_j)^*
\nonumber\\
&=q^{-1/2}t(n,n+1)^{\underline 0}_{\underline i}p_{rs}(u^2_j)^*\nonumber\\
&=v_-.\nonumber
\end{align}
Here we have used the following identities which are obtained from (\ref{action on basis}), (\ref{lr actions}) and
(\ref{Q iso}).
\begin{align}
&t(n,n+1)^{\underline 0}_{\underline i}\triangleleft K_1=t(n,n+1)^{\underline 0}_{\underline i},&&t(n,n+1)^{\underline 0}_{\underline i}\triangleleft E_1=0\nonumber\\
&p_{ij}\triangleleft E_1=0,&&(u^1_j)^*\triangleleft K_1=q^{1/2}(u^1_j)^*,\nonumber\\
&p_{ij}\triangleleft K_1=p_{ij},&&
(u^1_j)^*\triangleleft E_1=(-q)(u^2_j)^*.\nonumber
\end{align}
Similarly
\begin{align}
v_-\triangleleft E_1&=q^{-1/2}t(n,n+1)^{\underline 0}_{\underline i}p_{rs}(u^2_j)^*\triangleleft E_1\nonumber\\
&=q^{-1/2}\{(t(n,n+1)^{\underline 0}_{\underline i}\triangleleft E_1)
((p_{rs}(u^2_j)^*)\triangleleft K_1)\nonumber\\&+(t(n,n+1)^{\underline 0}_{\underline i}\triangleleft K_1^{-1})((p_{rs}(u^2_j)^*)\triangleleft E_1)\}\nonumber\\
&=q^{-1/2}t(n,n+1)^{\underline 0}_{\underline i}
 \{(p_{rs}\triangleleft E_1)( (u^2_j)^*\triangleleft K_1)\nonumber\\
&+
(p_{rs}\triangleleft K_1^{-1})((u^2_j)^*\triangleleft E_1)\}\nonumber\\
&=0.\nonumber
\end{align}
Two more identities which have been used above, are
\begin{align}
(u^2_j)^*\triangleleft K_1=q^{-1/2}(u^1_j)^*,&&(u^2_j)^*\triangleleft E_1=0.\nonumber
\end{align}

The case $(v_+,v_-)\triangleleft F_1=(0,v_+)$ is similar and the other two cases $(v_+,v_-)\triangleleft K_1=(q^{1/2}v_+,q^{-1/2}v_-)$ and
$(v_+,v_-)\triangleleft K_1K_2^2=q^{3/2}(v_+,v_-)$ are straightforward, but the following relations are needed.
\begin{align}
&t(n,n+1)^{\underline 0}_{\underline i}\triangleleft K_2=q^{1/2}t(n,n+1)^{\underline 0}_{\underline i},&
&t(n,n+1)^{\underline 0}_{\underline i}\triangleleft F_1=0,\nonumber\\
&(u^1_j)^*\triangleleft K_2=(u^1_j)^*,&&(u^2_j)^*\triangleleft K_2=q^{1/2}(u^2_j)^*,\nonumber\\
&(u^1_j)^*\triangleleft F_1=0,&&(u^2_j)^*\triangleleft F_1=(-q)^{-1}(u^1_j)^*,\nonumber\\
&p_{ij}\triangleleft K_2=p_{ij},&
&p_{ij}\triangleleft F_1=0.\nonumber
\end{align}
Case2. $\alpha \in$ Im $\phi_2$ is a general element.
\begin{align}
\alpha&=\phi_2(\sum_{n,\underline i}c_{n\underline i}t(n,n+1)^{\underline 0}_{\underline i}
 \otimes\sum_{r,s,j,k} d_{rsjk}p_{rs}\delb p_{jk})\nonumber\\
&=q^{-1/2}\sum_{n,\underline i,r,s,j,k}c_{n\underline i}t(n,n+1)^{\underline 0}_{\underline i}d_{rsjk}p_{rs}\binom{-q^{-3/2}(u^1_j)^*}{q^{-1/2}(u^2_j)^*}u^3_k\nonumber\\
&=q^{-1}\phi_1(\sum_k\{\sum_{\underline i,r,s,j}c_{\underline i}d_{rsjk}
t(n,n+1)^{\underline 0}_{\underline i}p_{rs}\binom{-q^{-3/2}(u^1_j)^*}{q^{-1/2}(u^2_j)^*}\}\otimes u^3_k)\nonumber\\
&=q^{-1}\phi_1(\sum_k A_k\otimes u^3_k),\nonumber
\end{align}
where
\begin{equation}
A_k=\sum_{n,\underline i,r,s,j}c_{n\underline i}d_{rsjk}
t(n,n+1)^{\underline 0}_{\underline i}p_{rs}\binom{q^{-3/2}
(u^1_j)^*}{q^{-1/2}(u^2_j)^*}\in \Omega^{(0,1)}.\nonumber
\end{equation}
The proof for $Im$ $\phi_2$ $\subset$ $Im$ $\phi_1$ is similar.\\
\end{proof}
In general the maps $\phi_1$ and $\phi_2$ will be defined as
\begin{align*}
&\phi_1:\Omega^{(0,1)}\otimes_{\mathcal A(\qp^1_q)}\mathcal L_N \rightarrow \mathcal A(SU_q(3))^2,\\
&\phi_1((v_+,v_-)^t\otimes \xi)=q^{\frac{N}{2}}(v_+\xi,v_-\xi)^t,\\
\end{align*}
and
\begin{align*}
&\phi_2:\mathcal L_N \otimes_{\mathcal A(\qp^1_q)}\Omega^{(0,1)} \rightarrow \mathcal A(SU_q(3))^2,\\
&\phi_2(\xi\otimes(v_+,v_-))=q^{-\frac{N}{2}}(\xi v_+,\xi v_-).
\end{align*}

Now, we prove that $\nabla^{\delb}_N$ has the right $\sigma$-twisted Leibniz property with respect to the map $\sigma=\phi_1^{-1} \phi_2$.
\begin{proposition}
Taking $\sigma$ as above, the following holds
\begin{eqnarray}\label{twisted Leibniz}
\nabla^{\delb}_N(\xi a)=(\nabla^{\delb}_N\xi) a+\sigma(\xi\otimes \delb a),
\quad \forall a \in \mathcal A(\qp^2_q),\,\forall\xi \in \mathcal L_N.
\end{eqnarray}
\end{proposition}
\begin{proof}
By (\ref{psi equalities 1}), (\ref{psi equalities 3}) and the fact that $\xi \triangleleft K_2=q^{N/2}\xi$, we compute the second component of the left hand side as follows
\begin{align}
q^{-N}\Psi_N^{\dagger}((\Psi_N \xi a)&\triangleleft F_2)\nonumber\\
&=q^{-N}\Psi_N^{\dagger}\{(\Psi_N\triangleleft F_2)((\xi a)\triangleleft K_2)+(\Psi_N\triangleleft K_2^{-1})((\xi a)\triangleleft F_2)\}\nonumber\\
&=q^{-N/2}(\xi \triangleleft F_2)a+q^{-N}\xi (a\triangleleft F_2).\nonumber
\end{align}
(Note that this actually is $\phi_1 \nabla_N^{\delb}$.)
For the second component of the right hand side we will get
\begin{eqnarray}
q^{-N/2}(\xi \triangleleft F_2) a+\sigma(\xi \otimes a\triangleleft F_2)\nonumber.
\end{eqnarray}
The previous lemma says that $q^{-N}$ will appear after acting $\sigma$ on the second term.
It can be seen that $\phi_1$ of both sides coincides.
 Computation for the second component will be similar.
\end{proof}
Now we will come up to the analog of proposition 3.8 of (\cite{KLS}).
\begin{proposition}\label{LN+M}
The tensor product connection
$\nabla_N^{\delb} \otimes 1+ (\sigma \otimes 1)(1 \otimes \nabla_M^{\delb})$
 coincides with the holomorphic structure on $\mathcal L_N\otimes_{\mathcal A(\qp^1_q)}\mathcal L_M$
when identified with $\mathcal L_{N+M}$.
\end{proposition}
\begin{proof}
\begin{eqnarray}
&&\nabla_{N+M}^{\delb}(\xi_1\xi_2)\nonumber \\
&&=q^{-(N+M)}\Psi^{\dagger}_{N+M}\delb\Psi_{N+M}(\xi_1\xi_2) \nonumber \\
&&=q^{-(N+M)}\Psi^{\dagger}_{N+M}\binom {(\Psi_{N+M}\xi_1\xi_2)\triangleleft F_2F_1}{(\Psi_{N+M}\xi_1\xi_2)
\triangleleft F_2}\nonumber\\
&&=q^{-(N+M)}\Psi^{\dagger}_{N+M}
\binom{\{(\Psi_{N+M}\triangleleft F_2) ((\xi_1\xi_2)\triangleleft K_2)\}\triangleleft F_1}{(\Psi_{N+M}\triangleleft F_2) ((\xi_1\xi_2)\triangleleft K_2)}\nonumber\\
&&+q^{-(N+M)}\Psi^{\dagger}_{N+M}\binom
{\{(\Psi_{N+M}\triangleleft K_2^{-1})( (\xi_1\xi_2)\triangleleft F_2)\}\triangleleft F_1}{(\Psi_{N+M}\triangleleft K_2^{-1})( (\xi_1\xi_2)\triangleleft F_2)}\nonumber\\
&&=q^{-\frac{N+M}{2}}\binom{(\xi_1\xi_2)\triangleleft F_2F_1}{(\xi_1\xi_2)\triangleleft F_2}
\nonumber\\&&= q^{-\frac{N}{2}}\binom{\{(\xi_1 \triangleleft F_2)\xi_2
+(q^{-N-M/2}\xi_1(\xi_2 \triangleleft F_2)\}\triangleleft F_1}{(\xi_1 \triangleleft F_2)\xi_2
+q^{-N-M/2}\xi_1(\xi_2 \triangleleft F_2)}.\nonumber
\end{eqnarray}
Besides (\ref{psi equalities 1}) and (\ref{psi equalities 2}), we also applied the identities
$\xi_i\triangleleft K_1=0$, $\xi_i\triangleleft F_1=0$.

On the other hand
\begin{eqnarray}
&&((\nabla_N^{\delb}\otimes 1)+(\sigma \otimes 1)(1 \otimes
\nabla_M^{\delb}))(\xi_1\otimes \xi_2)=\nonumber\\
&&q^{-N/2}\binom{\xi_1 \triangleleft F_2F_1}{\xi_1 \triangleleft F_2}\otimes \xi_2
+(\sigma \otimes 1)(\xi_1 \otimes q^{-M/2}
\binom{\xi_2 \triangleleft F_2F_1}{\xi_2 \triangleleft F_2}).\nonumber
\end{eqnarray}
Interpreting this expression as an element of $\Omega^{(0,1)}\otimes \mathcal L_{N+M}$,
 after applying the map $\sigma$, which gives us $q^{-N}$ on the second summand,
 we will get the same result.

\end{proof}

Thanks to proposition (\ref{LN+M}), the space $R:=\bigoplus H^0(\mathcal L_N, \nabla^{\delb}_N)$ has a ring structure under the natural tensor product of bimodules. In the following, we identify the quantum homogeneous coordinate ring $R$ with a twisted polynomial algebra in three variables 
\begin{theorem}\label{thm ring structure}
We have the algebra isomorphism
\begin{eqnarray}
R:=\bigoplus_{\substack N\geq 0}H^0(\mathcal L_N,\nabla_N^{\delb})\simeq
\frac{\compl \langle z_1,z_2,z_3\rangle}{\langle\, z_iz_j-qz_jz_i:1\leq i<j\leq 3\,\rangle}\nonumber
\end{eqnarray}
\end{theorem}
\begin{proof}
The ring structure on $R$ is coming from the tensor product $\mathcal L_{N_1}\otimes_{\mathcal A(\qp^2_q)} \mathcal L_{N_2} \simeq \mathcal L_{N_1+N_2}$. The following discussion shows that $H^0(\mathcal L_1,\nabla_1^{\delb})=\mathbb Cz_1
\oplus\mathbb Cz_2\oplus\mathbb Cz_3$. In order to do this, we will give an explicit formula for the basis elements of $H^0(\mathcal L_N,\nabla _N^{\delb})$, $t(0,N)^0_{\underline j}$.

Let us look at the computation more closely.
Using (\ref{t basis}), we will see that
\begin{align}
t(0,N)^{\underline 0}_{\underline j}=[j_2+1]\sqrt{\frac{[\frac{j_2}{2}+m]![N-j_2]!}{[\frac {j_2}{2}-m]![N]!}}F_1^{1/2j_2-m} F_2^{j_2}\triangleright z_3^N.
\end{align}
By induction it is not difficult to prove that
$F_2\triangleright z_3^N=q^{-\frac{N-1}{2}}[N]z_2z_3^{N-1}$. Therefore
\begin{align}
F_2^j\triangleright z_3^N&=q^{-\frac{N-1}{2}-\frac{N-2}{2}-...-\frac{N-j}{2}+\frac{1}{2}+...+\frac{j-1}{2}}
[N]...[N-j+1]z_2^jz_3^{N-j}\nonumber\\&=
q^{\frac{j^2}{2}-\frac{jN}{2}}\frac{[N]!}{[N-j]!}z_2^jz_3^{N-j}.\nonumber
\end{align}
The same method gives
\begin{align}
F_1^r\triangleright z^j=q^{\frac{r^2}{2}-\frac{rj}{2}}\frac{[j]!}{[j-r]!}z_1^rz_2^{j-r}.\nonumber
\end{align}
So
\begin{align}F_1^rF_2^j\triangleright z_3^N=q^{\frac{j^2}{2}-\frac{jN}{2}+\frac{r^2}{2}-\frac{rj}{2}}
\frac{[N]!}{[N-j]!}\frac{[j]!}{[j-r]!}z_1^rz_2^{j-r}z_3^{N-j}.\nonumber
\end{align}
Replacing
$j$ with $j_2$ and $r$ with $1/2j_2-m$, we will have
\begin{eqnarray}
t(0,N)^{\underline 0}_{\underline j}=[j_2+1]!\sqrt{\frac{[N]!}{[\frac{j_2}{2}-m]![\frac{j_2}{2}+m]![N-j_2]!}}
q^{\alpha}
z_1^{1/2j_2-m}z_2^
{1/2j_2+m}z_3^{N-j_2},\nonumber
\end{eqnarray}
where $\alpha=-\frac{{j_2N}}{2}-{(\frac{j_2}{2}-m)}\frac {j_2}{2}+\frac{j_2^2}{2}+
\frac{1}{2}{(\frac {j_2}{2}-m)^2}$.

In the case $N=1$, $t(0,1)^{\underline 0}_{0,1,-\frac{1}{2}}=[2]z_1$, $t(0,1)^{\underline 0}_{0,1,\frac{1}{2}}=q[2]z_2$ and $t(0,1)^{\underline 0}_{\underline 0}=z_3$.
Now the isomorphism follows from the identities
$z_i\otimes_{\mathcal A(\qp^2_q)}z_j-qz_j\otimes_{\mathcal A(\qp^2_q)}z_i=0$ in $\mathcal L_2$, which can easily be seen.\\
\end{proof}

\section{The $C^*$-algebras $C(SU_q(3))$ and $C(\qp^2_q)$ }\label{holcts}
In this section we extend   the results of Proposition (\ref{holfunprop}) and Theorem (\ref{holsectthm}) which are stated for polynomial functions and polynomial sections to $L^2$-functions and sections, respectively.

Let $C(SU_q(3))$ denotes the $C^*$ completion of $\mathcal A(SU_q(3))$, i.e. the universal $C^*$-algebra generated by the elements $u^i_j$ subject to the relations given in section \ref{qgsuq3}. There exists a unique left invariant normalized Haar state on this compact quantum group denoted by $h$. The functional $h$ is faithful and it also has a twisted tracial property which will be considered in the next section. If we denote the Hilbert space of completion of $\mathcal A(SU_q(3))$ with respect to the inner product $\langle a,\,b\rangle:=h(a^*b)$ by $L^2(SU_q(3))$. Since the Haar state on the $C^*$-algebra $C(S U_q(3))$ is faithfull \cite{N}, the GNS map $\eta:C(SU_q(3))\rightarrow L^2(SU_q(3))$ will be injective. An orthogonal basis of $L^2(SU_q(3))$, would be $\eta(t(n_1,n_2)^{\underline l}_{\underline j})$.

Using  the fact that $L^2(SU_q(3))$ is the completion of $\bigoplus_{(n_1,n_2)\in \mathbb{N}^2}V(n_1,n_2)\otimes V(n_1,n_2)$, the action of $U_q(\mathfrak{su}(3))$ naturally rises to an action on $L^2(SU_q(3))$. The invariant subalgebra of $C(SU_q(3))$ under the action of $U_q(\mathfrak{u}(2))$ is by definition the $C^*$-algebra  $C(\qp^2_q)$. By the  invariance property of the Haar state,  the  GNS map restricts  to an injective map  $C(\qp^2_q)\rightarrow L^2(SU_q(3))^{U_q(\mathfrak{u}(2))}$.
The space of continuous sections and $L^2$-sections can be defined as well.
\begin{align*}
\Gamma(\mathcal L_N):&=\{\xi\in C(SU_q(3))|\xi\triangleleft k=\epsilon(k)\xi,\,\xi\triangleleft K_1K_2^2=q^N\xi, \quad\forall k\in U_q(\mathfrak{u}(2))\}\\
L^2(\mathcal L_N):&=\{\xi\in L^2(SU_q(3))|\xi\triangleleft k=\epsilon(k)\xi,\,\xi\triangleleft K_1K_2^2=q^N\xi, \quad\forall k\in U_q(\mathfrak{u}(2))\}\\
&=\text{Span}\{t(n,n+N)^{\underline 0}_{\underline j}|\,\,n\in \mathbb{N},\, \,\, \underline j\, \text{satisfies}\, (\ref{j indices})\,\}^{closure}
\end{align*}
Note the the last equality is for $N\geq 0$. For $N<0$, basis elements are of the form  $t(n-N,n)^{\underline 0}_{\underline j}$.

The operator $Z=\triangleleft(F_2F_1,F_2)$ is unbouded on $L^2(SU_q(3))$, so we have to specify the domain of this operator.
\begin{equation*}
\text{Dom} (Z):=\{a \in L^2(SU_q(3))|\,(a\triangleleft F_2F_1,a\triangleleft F_2)\in L^2(SU_q(3)^2)\}.
\end{equation*} 
Now the Proposition \ref{holfunprop} can easily be generalized to the following proposition.
\begin{proposition}
The Kernel of the  map $Z$ restricted to $L^2(\qp^2_q)$ is $\mathbb{C}$.
\end{proposition}
\begin{proof}
Since any element of $L^2(\qp^2_q)$ is a $L^2$-linear combination of the elements $t(n,n)^{\underline 0}_{\underline j}$, proof is exactly like Proposition \ref{holfunprop}.
\end{proof}
Let us  define $\text{Dom} (\delb):=\{a\in C(\qp^2_q)| \,||\delb a||<\infty\}$. The above statement could pass to continuous functions as follows.
\begin{corollary}
There is no non-constant holomorphic  function in $C(\qp^2_q)$. 
\end{corollary}
With a similar discussion, the analog of \ref{holsectthm} continues to hold if we work with $L^2$-sections of $\mathcal L_N$. We give the statement of the theorem and leave its  similar proof to the reader.
\begin{theorem}
Let N be a positive integer. Then
\begin{eqnarray}
(1) &&\qquad H^0(L^2(\mathcal L_N),\nabla_N^{\delb})\simeq\mathbb C^{\frac{(N+1)(N+2)}{2}},\nonumber\\
(2) &&\qquad H^0(L^2(\mathcal L_{-N}),\nabla_{-N}^{\delb})=0.\nonumber
\end{eqnarray}
\end{theorem}

We note that our approach here as well as in \cite{KLS}, is somehow the opposite of the approach adopted in \cite{Av1, Av2} to noncommutative projective spaces. We started with a $C^*$-algebra defined as the quantum homogeneous space of the quantum group $SU_q(3)$ and its natural line bundles, and  endowed them with holomorphic structures. The quantum homogeneous coordinate ring  is then defined as the algebra of holomorphic sections of these line bundles. This ring coincides with the twisted homogeneous ring associated in \cite{Av1, Av2} to the line bundle $\mathcal{O}(1)$ under a suitable twist.

\section{Existence of a twisted positive Hochschild 4-cocycle on $\qp^2_q$}
In \cite{C1},  Section VI.2, Connes shows 
that extremal positive Hochschild cocycles on the algebra of smooth functions on a compact oriented
2-dimensional manifold encode the information needed to define a holomorphic
structure on the surface. There is
a similar result for holomorphic structures on the noncommutative two torus (cf. {\it Loc cit.}).
In particular the positive Hocshchild cocycle is defined via the holomorphic structure
and represents the fundamental cyclic  cocycle. In \cite{KLS} a notion of twisted positive Hochschild cocycle is
 introduced
and a similar result is proved for the  holomorphic structure of $\qp^1_q$.
Although the corresponding problem of characterizing holomorphic
structures on higher dimensional (commutative or noncommutative) manifolds via positive Hochschild cocycles is still
open, nevertheless these results suggest regarding (twisted) positive Hochschild cocycles as a possible
framework for holomorphic noncommutative structures. In this section we
prove an analogous result for $\qp^2_q$.

First we recall the notion of twisted Hochschild and cyclic cohomologies.
Let $\mathcal A$ be an algebra and $\sigma$ an automorphism of $\mathcal A$.
For each $n\geq 0$, $C^n(\mathcal A):=$ Hom$(\mathcal A^{\otimes(n+1)},\mathbb C)$
is the space of \textit{n-cochains} on $\mathcal A$.
Define the space of \textit {twisted Hochschild n-cochains} as
$C^n_{\sigma}(\mathcal A):=$Ker$\{(1-\lambda_{\sigma}^{n+1}):C^n(\mathcal A)\rightarrow C^n(\mathcal A)\}$,
where the \textit{twisted cyclic} map $\lambda_{\sigma}:C^n(\mathcal A)\rightarrow C^n(\mathcal A)$ is defined as
\begin{eqnarray}
(\lambda_{\sigma}\phi)(a_0,a_1,...,a_n)=(-1)^n \phi(\sigma(a_n),a_0,a_1,...,a_{n-1}).\nonumber
\end{eqnarray}

The \textit {twisted Hochschild coboundary} map $b_{\sigma}:C^n(\mathcal A)\rightarrow C^{n+1}(\mathcal A)$ is given by
\begin{align*}
b_{\sigma}\phi(a_0,a_1,...,a_{n+1})=&
\sum_{i=0}^n(-1)^i\phi(a_0,...,a_ia_{i+1},...,a_{n+1})\\
&+(-1)^{n+1}\phi(\sigma(a_{n+1})a_0,...,a_n).\nonumber
\end{align*}
The cohomology of the complex $(C^*_{\sigma}(\mathcal A),b_{\sigma})$ is called
 the \textit
{twisted Hochschild cohomology} of $\mathcal A$.
We also need the notion of \textit {twisted cyclic cohomology} of $\mathcal A$.
It is by definition the cohomology of the complex
$(C^*_{\sigma,\lambda}(\mathcal A),b_{\sigma})$,
where
\begin{eqnarray}
C^n_{\sigma,\lambda}:=Ker\{(1-\lambda):C^n_{\sigma}(\mathcal A)\rightarrow
C^{n+1}_{\sigma}(\mathcal A)\}.\nonumber
\end{eqnarray}

Now we come back to the case of our interest, that is $\qp^2_q$.
Let $\tau$ be the fundamental class on $\mathbb{C}P^2_q$ defined as in \cite{DL} by
\begin{equation}\label{tau}
\tau(a_0,a_1,a_2,a_3,a_4):=-\int_h a_0 \ud a_1\ud a_2\ud a_3\ud a_4\,,\quad\forall a_0, a_1,...,a_4 \in \mathcal A(\qp^2_q).
\end{equation}
Here $h$ stands for the Haar state functional of the quantum group $\mathcal A(SU_q(3))$
which has a twisted tracial property $h(xy)=h(\sigma(y)x)$. Here the algebra automorphism $\sigma$
is defined by
\begin{equation}
\sigma:\mathcal A(SU_q(3))\rightarrow \mathcal A(SU_q(3)), \quad \sigma(x)=K\triangleright x\triangleleft K.\nonumber
\end{equation}
where $K=(K_1K_2)^{-4}$. The map $\sigma$, restricted to the algebra $\mathcal A(\qp^2_q)$ is given by
$\sigma(x)=K\triangleright x$. Non-triviality of $\tau$ has been shown in \cite{DL}. Now we recall the definition of a twisted positive Hochschild cocycle as given in \cite{KLS}.
\begin{definition}
A twisted Hochschild 2n-cocycle $\phi$ on a $\ast$-algebra $\mathcal A$ is said to be
twisted positive if the following map defines a positive sesquilinear form on the vector space $\mathcal A^{\otimes(n+1)}$:
\begin{eqnarray}
\langle a_0\otimes a_1\otimes...\otimes a_n,b_0\otimes b_1\otimes...\otimes b_n\rangle=
\phi(\sigma(b_n^*)a_0,a_1,...,a_n,b_n^*,...,b_1^*).\nonumber
\end{eqnarray}
\end{definition}
We would like to define a twisted Hochschild cocycle $\varphi$ which is cohomologous to $\tau$ and it is positive. For simplicity, we introduce first the maps $\varphi_i$, for $i=1,2$ as follows
\begin{align}\label{phi1,2}
&\varphi_1(a_0,a_1,a_2,a_3,a_4)=-3\int_h a_0\partial a_1 \partial a_2 \delb a_3 \delb a_4\nonumber,\\
&\varphi_2(a_0,a_1,a_2,a_3,a_4)=-3\int_h a_0\delb a_1 \delb a_2 \partial a_3 \partial a_4.
\end{align}
Now we define $\varphi\in C^4(\mathcal A(\qp^2_q))$ by
\begin{equation}\label{phi}
\varphi:=\varphi_1+\varphi_2\,.
\end{equation}

We will need the following simple lemma for future computations.
\begin{lemma}
For any $a_0,a_1,a_2,a_3,a_4,a_5$ $\in \mathcal A(\qp^2_q)$ the following identities hold:
\begin{align*}
\int_ha_0(\partial a_1\partial a_2\delb a_3\delb a_4) a_5=\int_h\sigma(a_5)a_0\partial a_1\partial a_2\delb a_3\delb a_4,\\
\int_ha_0(\delb a_1\delb a_2\partial a_3\partial a_4) a_5=\int_h\sigma(a_5)a_0\delb a_1\delb a_2\partial a_3\partial a_4.
\end{align*}
\end{lemma}
\begin{proof}
We give the proof of the first one. The proof for the second equality will be similar.
The space of $\Omega^{(2,2)}$ is a rank one free $\mathcal A(\qp^2_q)$-module. Let $\omega$ be the central basis element for the space of $\Omega^{(2,2)}$
and let $\partial a_1\partial a_2\delb a_3\delb a_4=x\omega$. Then
\begin{align}
\int_ha_0(\partial a_1\partial a_2\delb a_3\delb a_4)a_5-\int_h\sigma(a_5)a_0\partial a_1\partial a_2\delb a_3\delb a_4
=\int_h(a_0 x\omega a_5-\sigma(a_5)a_0x\omega)&\nonumber\\=\int_h(a_0 x a_5\omega-\sigma(a_5)a_0x\omega )&\nonumber\\
=h(a_0 xa_5-\sigma(a_5)a_0x)=0&.\nonumber
\end{align}
The last equality comes from the twisted property of the Haar state.
\end{proof}
\begin{proposition}
The functional $\varphi$ defined by formula (\ref{phi}), is a twisted positive Hochschild 4-cocycle.
\end{proposition}
\begin{proof}
We first verify the twisted cocycle property. In order to do so, we consider this property for each $\varphi_i$. We will prove the statement for $\varphi_1$. The proof for $\varphi_2$ is similar.
\begin{align}
&\varphi_1(\sigma(a_0),\sigma(a_1),\sigma(a_2),\sigma(a_3),\sigma(a_4))\nonumber\\
&=-3\int_h \sigma(a_0) \partial \sigma (a_1) \partial \sigma (a_2)
\delb \sigma (a_3) \delb \sigma (a_4)\nonumber\\
&=-3\int_h (K \triangleright a_0) (K \triangleright \partial a_1)
( K \triangleright \partial a_2)(K \triangleright \delb a_3)(K \triangleright \delb a_4)\nonumber\\
&=-3\int_h K \triangleright (a_0\partial a_1 \partial a_2
\delb a_3 \delb a_4)= -3\,\epsilon(K)\int_h a_0\partial a_1
\partial a_2 \delb a_3 \delb a_4\nonumber\\
&=\varphi_1(a_0,a_1,a_2,a_3,a_4)\nonumber.
\end{align}

Now let us prove that $b_{\sigma}\varphi=0$. Again we just prove for  $\varphi_1$ and leave the similar proof of the other one.
\begin{align*}
b_{\sigma}\varphi_1(a_0,a_1,a_2,a_3,a_4,a_5)&=\varphi_1(a_0a_1,a_2,a_3,a_4,a_5)
-\varphi_1(a_0,a_1a_2,a_3,a_4,a_5)\\&+\varphi_1(a_0,a_1,a_2a_3,a_4,a_5)-\varphi_1(a_0,a_1,a_2,a_3a_4,a_5)\\&+\varphi_1(a_0,a_1,a_2,a_3,a_4a_5)-\varphi_1(\sigma(a_5)a_0,a_1,a_2,a_3,a_4)\\
\end{align*}
Using (\ref{phi1,2}), this equals to
\begin{align*}
-3&\int_ha_0a_1\partial a_2\partial a_3\delb a_4\delb a_5+3\int_ha_0\partial(a_1a_2)\partial a_3\delb a_4\delb a_5
\\-3&\int_ha_0\partial a_1\partial (a_2a_3)\delb a_4\delb a_5+3\int_ha_0\partial a_1\partial a_2\delb (a_3a_4)\delb a_5\\
-3&\int_ha_0\partial a_1\partial a_2\delb a_3\delb (a_4a_5)+3\int_h\sigma(a_5)a_0\partial a_1\partial a_2\delb a_3\delb a_4.
\end{align*}
Using the Leibniz property we get
\begin{align*}
b_{\sigma}\varphi_1(a_0,a_1,a_2,a_3
,a_4,a_5)=-3\int_h(a_0a_1\partial a_2\partial a_3\delb a_4\delb a_5-\sigma(a_5)a_0\partial a_1\partial a_2\delb a_3\delb a_4),
\end{align*}
which is zero by the previous lemma.

Now we will show that all $\varphi_1$ and $\varphi_2$ are positive.\\\\
Positivity of $\varphi_1$:
\begin{align}
\varphi_1(\sigma (a_0^*)a_0,a_1,a_2,a_2^*,a_1^*)&=
-3\int_h \sigma(a_0^*)a_0 \partial a_1 \partial a_2 \delb a_2^* \delb a_1^*\nonumber\\
&=-3\int_ha_0 \partial a_1 \partial a_2 \delb a_2^* \delb a_1^*a_0^*\nonumber\\
&=3\int_h(a_0 \partial a_1 \partial a_2)(a_0 \partial a_1 \partial a_2)^*.\nonumber
\end{align}
One can take $\partial a_1=(v_1,v_2)$ and $\partial a_2 =(w_1,w_2)$, then using the multiplication rule of type (1,0) forms (c.f. \cite{DL} Proposition A.1), we find that\\
$(a_0 \partial a_1 \partial a_2)(a_0 \partial a_1 \partial a_2)^*=c_4^2[2]^{-1}\mu\mu^*$,
where $\mu=q^{1/2}a_0 v_1w_2-q^{-1/2}a_0v_2w_1$. Hence
\begin{equation}
\varphi_1(\sigma (a_0^*)a_0,a_1,a_2,a_2^*,a_1^*)=h(3c_4^2[2]^{-1}\mu\mu^*)\geq 0.\nonumber
\end{equation}
\\\ Positivity of $\varphi_2$:
\begin{align}
\varphi_2(\sigma (a_0^*)a_0,a_1,a_2,a_2^*,a_1^*)&=
 -3\int_h \sigma(a_0^*)a_0 \delb a_1 \delb a_2 \partial a_2^* \partial a_1^*\nonumber\\
&=-3\int_ha_0 \delb a_1 \delb a_2 \partial a_2^* \partial a_1^*a_0^*\nonumber\\
&=3\int_h(a_0 \delb a_1 \delb a_2)(a_0 \delb a_1 \delb a_2)^*.\nonumber
\end{align}

Similar to the above discussion, one can take $\delb a_1=(v_1,v_2)$ and $\delb a_2 =(w_1,w_2)$ and use the multiplication of type (0,1) forms to find that
\begin{align}
\varphi_2(\sigma (a_0^*)a_0,a_1,a_2,a_2^*,a_1^*)=
h(3c_0^2[2]^{-1}\nu\nu^*)\geq 0,\nonumber
\end{align}
where $\nu=q^{1/2}a_0 v_1w_2-q^{-1/2}a_0v_2w_1$.
Here $c_0$ and $c_4$ are two real constants.
This concludes the positivity of $\varphi$.\\
\end{proof}
Now we want to show that the twisted Hochschild cocycle $\varphi$ as defined by formula (\ref{phi}) and the twisted cyclic cocycle $\tau$ as in formula (\ref{tau})
 are cohomologous. To this end, we need an appropriate twisted Hochschild cocycle
$\psi$ such that $\tau-\varphi=b_{\sigma}\psi$.
Let $\psi_i$ for , $i$=1,2,3,4 be defined by
\begin{align}
\psi_1(a_0,a_1,a_2,a_3)=
-&\int_h a_0\partial{a_1}\overline{\partial}{a_2}\partial\overline{\partial}a_3,\nonumber\\
\psi_2(a_0,a_1,a_2,a_3)=2&\int_ha_0\partial{a_1}\partial\overline{\partial}{a_2}\delb a_3\nonumber,\\
\psi_3(a_0,a_1,a_2,a_3)=2&\int_h a_0\delb{a_1}\delb\partial{a_2}\partial{a_3},\nonumber\\
\psi_4(a_0,a_1,a_2,a_3)=
-&\int_h a_0\delb{a_1}\partial{a_2}\delb\partial{a_3}.\nonumber
\end{align}
and let $\psi=\sum_{i=1}^4 \psi_i$. Then we will have the following result.
\begin{proposition}
The twisted Hochschild cocycles $\tau$ and $\varphi$ are cohomologous.
\end{proposition}
\begin{proof}
\begin{align}
b_{\sigma}\psi_1(a_0,a_1,a_2,a_3,a_4) & =\psi_1(a_0a_1,a_2,a_3,a_4)
-\psi_1(a_0,a_1a_2,a_3,a_4)\nonumber\\
 & +\psi_1(a_0,a_1,a_2a_3,a_4)
-\psi_1(a_0,a_1,a_2,a_3a_4)\nonumber\\
 & +\psi_1(\sigma(a_4)a_0,a_1,a_2,a_3)\nonumber
\end{align}
which equals to
\begin{align}
-\int_h \{ & a_0 a_1\partial a_2 \delb a_3 \partial\delb a_4
-a_0\partial(a_1a_2)\delb a_3 \partial \delb a_4 +
a_0\partial a_1\delb (a_2a_3) \partial\delb a_4\nonumber\\
 & -a_0\partial a_1\delb a_2 \partial\delb (a_3a_4) +
\sigma(a_4)a_0\partial a_1\delb a_2 \partial\delb a_3\}.\nonumber
\end{align}

Applying the Leibniz rule, one can see that in the expanded form, all but two terms will cancel. That is
\begin{align}
b_{\sigma}\psi_1 & =\int_h a_0(\partial a_1 \delb a_2 \partial a_3 \delb a_4 - \partial a_1 \delb a_2 \delb a_3 \partial a_4). \nonumber
\end{align}
Similar computation for $\psi_i$, $i=2,3$ and $4$ shows that
\begin{align}
b_{\sigma}\psi_2 & =2\int_h a_0(\partial a_1 \partial a_2 \delb a_3 \delb a_4 - \partial a_1 \delb a_2 \partial a_3 \delb a_4) \nonumber,\\
b_{\sigma}\psi_3 & =2\int_h a_0(\delb a_1 \delb a_2 \partial a_3 \partial a_4 - \delb a_1 \partial a_2 \delb a_3 \partial {a_4}) \nonumber,\\
b_{\sigma}\psi_4 & =\int_h a_0(\delb a_1 \partial a_2 \delb a_3 \partial a_4 - \delb a_1 \partial a_2 \partial a_3 \delb a_4). \nonumber
\end{align}
Therefore
\begin{align}\label{bpsi}
b_{\sigma}\psi=
2&\int_h a_0(\partial a_1 \partial a_2 \delb a_3 \delb a_4+\delb a_1 \delb a_2 \partial a_3 \partial a_4) \nonumber\\
-&\int_h a_0(\delb a_1 \partial a_2 \partial a_3 \delb a_4 +\partial a_1 \delb a_2 \delb a_3 \partial a_4)\nonumber\\
-&\int_ha_0(\delb a_1 \partial a_2 \delb a_3 \partial a_4+\partial a_1 \delb a_2 \partial a_3 \delb a_4).
\end{align}
Now from (\ref{tau}), (\ref{phi}) and (\ref{bpsi}), we can easily find that $\tau-\varphi=b_{\sigma}\psi$.
\end{proof}
\subsection*{Acknowledgments}
The first author would like to thank Giovanni Landi and Walter
van Suijlekom for many discussions on the topic of this paper as well the collaboration in \cite{KLS} which is
naturally continued in  the present work.  The second author is much obliged and thankful to
Francesco D'Andrea for kindly and promptly  answering many questions
about the subject of \cite{DL,DDL}.

Department of Mathematics, University of Western Ontario, London, Ontario, N6A5B7, Canada.

\textit{Email}: masoud@\,uwo.ca\\\\
Department of Mathematics, University of Western Ontario, London, Ontario, N6A5B7, Canada.

\textit{Email}: amotadel@\,uwo.ca
\end{document}